\documentclass[10pt,leqno, psamsfonts]{amsart}

\usepackage{indentfirst,csquotes}

\usepackage{paralist,fancyhdr,etoolbox}
\usepackage{amsmath, amsthm, amscd, amsfonts, amssymb, graphicx, color}
\usepackage[utf8]{inputenc}
\usepackage{bm}
\usepackage[most]{tcolorbox}
\usepackage{xcolor}
\usepackage[colorlinks,citecolor=blue]{hyperref}
\usepackage{orcidlink}
\usepackage{cleveref}
\usepackage{mathrsfs}         
\usepackage{enumitem}         
\usepackage{graphicx}         
\usepackage{tikz}             
\usepackage{nicematrix}
\usepackage{caption}

\usepackage[textsize=normalsize]{todonotes}

\usetikzlibrary{calc, arrows.meta}
\definecolor{arrowcyan}{RGB}{0, 191, 255}
\definecolor{arrowblue}{RGB}{0, 0, 200}
\definecolor{labelgreen}{RGB}{0, 180, 0}

\newtheorem{theorem}{Theorem}[section]
\newtheorem{lemma}[theorem]{Lemma}

\newtheorem{corollary}[theorem]{Corollary}

\newtheorem*{question*}{Question}
\newtheorem*{theorem*}{Theorem}

\theoremstyle{definition}
\newtheorem{definition}[theorem]{Definition}

\newtheorem{remark}[theorem]{Remark}

\newcommand{\abs}[1]{|#1|}
\newcommand{\less}{<}
\newcommand{\more}{>}
\makeatletter
\def\moverlay{\mathpalette\mov@rlay}
\def\mov@rlay#1#2{\leavevmode\vtop{%
   \baselineskip\z@skip \lineskiplimit-\maxdimen
   \ialign{\hfil$\m@th#1##$\hfil\cr#2\crcr}}}
\newcommand{\charfusion}[3][\mathord]{
    #1{\ifx#1\mathop\vphantom{#2}\fi
        \mathpalette\mov@rlay{#2\cr#3}
      }
    \ifx#1\mathop\expandafter\displaylimits\fi}
\makeatother

\newcommand{\cupdot}{\charfusion[\mathbin]{\cup}{\cdot}}

\DeclareMathOperator{\out}{out}
\DeclareMathOperator{\ind}{in}
\DeclareMathOperator{\charg}{charge}
\DeclareMathOperator{\Ind}{IN}
\DeclareMathOperator{\Out}{OUT}

\begin{document}
\title{Cofinality of Regular Tournaments} 

\author[O. Hatem]{Omar Hatem}
\address{Omar Hatem \\ Department of Mathematics and Actuarial Science \\The American University in Cairo\\
 Egypt}
\email{omarhatem2002@aucegypt.edu}
\urladdr{}{}

\author[S. Mohamed]{Sara Mohamed}
\address{Sara Mohamed \\ Department of Mathematics and Actuarial Science \\The American University in Cairo\\
 Egypt}
\email{sara\textunderscore mohamed@aucegypt.edu}

\author[I. M\"uller]{Isabel M\"uller}
\address{Isabel M\"uller \\ Department of Mathematics and Actuarial Science \\
The American University in Cairo \\ Egypt }
\email{isabel.muller@aucegypt.edu}
\urladdr{\href{https://sites.google.com/view/isabelmuller/}{https://sites.google.com/view/isabelmuller/}}

\author[D. Siniora]{Daoud Siniora}
\address{Daoud Siniora \\ Department of Mathematics and Actuarial Science \\The American University in Cairo\\
 Egypt}
\email{daoud.siniora@aucegypt.edu}
\urladdr{\href{https://sites.google.com/view/daoudsiniora/}{https://sites.google.com/view/daoudsiniora/}}

\date{\today}

\maketitle

\let\thefootnote\relax
\footnotetext{MSC2010: Primary: 05C20, Secondary 05C85.} 

\begin{abstract}
We show that the class of all finite regular tournaments is cofinal in the class of finite tournaments. In addition, we establish cofinality results for certain special subclasses of regular tournaments. We also provide an algorithm for constructing these regular tournaments. 
\end{abstract} 

\section{Introduction}

Tournaments are directed graphs with no loops that have exactly one oriented edge between any pair of distinct vertices. A tournament is obtained by choosing one direction for every edge in an undirected complete graph. This paper deals only with finite tournaments. They arise naturally in combinatorics, social choice theory, and the study of dominance relations, where an edge $u \to v$ encodes that $u$ dominates or defeats $v$, as is the case in an actual (sports) tournament where all participants compete against one another. From a structural perspective, tournaments exhibit a remarkably rich theory.

A very fruitful general principle of mathematical research stems from so-called local-global principals, which investigate how 
global structural properties of larger structures emerge from purely local constraints. Classical results illustrate this phenomenon: every tournament contains a directed Hamiltonian path \cite{hamiltionpath}, where Hamiltonianity is often established by studying local configurations. Relatedly, there is Camion's Theorem \cite{camion} which states that if for any partition $A\cup B$ of the vertex set, there is an arrow from $A$ into $B$, then the tournament contains a directed Hamiltonian cycle.  The present paper focuses on regular tournaments, i.e. tournaments in which every vertex has equal indegree and outdegree. These form highly symmetric objects supporting strong cycle structure. For instance, they contain cycles of any length \cite{alspach}. Moreover, tournaments appear naturally in connection with algebraic objects: the automorphism group of a tournament has odd order, and conversely, every finite group of odd order arises as the automorphism group of some tournament \cite{moon}, highlighting their expressive combinatorial power.

Among all tournaments, regular tournaments occupy a  special role. They may be viewed as the “balanced” objects in the class: every vertex has the same number of wins and losses, or equivalently is of zero charge. This balancing condition forces strong global uniformity and symmetry, while still allowing a large diversity of configurations. In particular, regular tournaments exist only in odd orders, and their score sequence is completely uniform, in contrast to arbitrary score sequences permitted by general tournaments, which are classified by Landau's Theorem \cite{landau}. From a structural standpoint, it is therefore natural to ask the following question.

\begin{question*}
   How frequently do regular tournaments appear in the class of all tournaments?
\end{question*}
The main result of this paper answers shows that regular tournaments are rather ubiquitous in the class of tournaments.

\begin{theorem*}
    Every finite tournament embeds into a finite regular tournament. 
\end{theorem*} 

This means that regular tournaments form a cofinal subclass in the class of tournaments, which implies that there is no loss of generality by restricting ones attention to them when studying hereditary properties. Any local imbalance present in a tournament can be absorbed by extending the structure while restoring global balance.

We further refine this phenomenon by identifying two natural subclasses of regular tournaments, each defined by a symmetry condition at a distinguished vertex. Let $v$ be a vertex and denote by $\Ind(v)$ and $\Out(v)$ its in-neighbourhood and out-neighbourhood, respectively. We say that a regular tournament is of \textit{Type-I} if there exists some vertex $v$ such that the induced subtournament on $\Ind(v)$ is isomorphic to the one induced on $\Out(v)$. We say that a regular tournament is of \textit{Type-II} if there exists some $v$ such that the induced subtournament on $\Ind(v)$ is isomorphic to the inversion of that on $\Out(v)$.

These conditions can be interpreted as local symmetry principles: in Type-I, the incoming and outgoing subtournaments of some vertex are structurally identical, while in Type-II they are dual to each other. Our second main result shows that these symmetry constraints still preserve universality.

\begin{theorem*}
    Both Type-I and Type-II regular tournaments are cofinal in the class of all finite tournaments.
\end{theorem*}  

This is somewhat surprising: even after imposing strong local symmetry at a vertex, one still retains enough flexibility to embed arbitrary tournaments. From a broader perspective, these results contribute to a general paradigm in combinatorics and model-theoretic graph theory: identifying highly structured subclasses that remain universal for embeddings. Such classes often serve as canonical environments in which arbitrary configurations can be studied while benefiting from additional symmetry and regularity. In particular, cofinal subclasses are natural candidates for generic or universal constructions, and may be useful in probabilistic, extremal, or model-theoretic analysis of tournaments.

Finally, we complement our existence results with explicit constructions, including an inductive procedure for building regular extensions and a matrix-based construction via Gale–Ryser \cite{gale, ryser} for the Type-II case. We also propose a greedy algorithm for constructing Type-II regular tournaments and verify it computationally for small sizes.


\section{Preliminaries}
\begin{definition}
A \textit{tournament} $T$ is a pair $(V,E)$ of sets where $E\subseteq V\times V$ such that $E$ is an irreflexive, antisymmetric, and total binary relation on $V$. 
\end{definition}

We will use $T$ and $V$ interchangeably. Observe that $(V,E)$ is a tournament if and only if for any two distinct vertices $u,v\in V$, either $(u,v)\in E$ or $(v,u)\in E$, but not both. As usual, members of $V$ are called \textit{vertices}, and members of $E$ are called \textit{directed edges} or \textit{arrows}. The number of vertices is the \textit{order} of the tournament. We write $u\to v$ when $(u,v)\in E$.  

Given a vertex $v\in V$, we define the subsets $\Ind(v):=\{u\in V\mid u\to v\}$ and $\Out(v):=\{u\in V\mid v\to u\}$. The \textit{indegree of $v$}, denoted by $\ind(v)$, is the number of incoming edges to $v$, whence clearly, $\ind(v)=\abs{\Ind(v)}$. In the same vein, the \textit{outdegree of $v$}, denoted by $\out(v)$, is the number of outgoing edges of $v$, and so $\out(v)=\abs{\Out(v)}$. Clearly, $\ind(v)+\out(v)=|V|-1$ for all vertices $v\in V$. 

The \textit{score sequence} of a tournament is defined to be the sequence of the outdegrees of all vertices, written in nondecreasing order. More precisely, it is the sequence $(s_1,s_2,\ldots,s_n)$, where $V=\{v_1,\ldots,v_n\} $, $s_i=\out(v_i)$, and $s_1\leq s_2\leq \cdots \leq s_n$. A tournament $(V,E)$ where the edge relation $E$ is also transitive is called a \textit{transitive tournament}. The score sequence of a transitive tournament of order $n$ is $(0,1,2,\ldots,n-1)$.

The \textit{charge} of a vertex $v$ is defined to be the difference between its indegree and outdegree, that is, $\charg(v)= \ind(v)-\out(v).$ 
When $\charg(v)=0$, we say $v$ has a \textit{neutral charge}.

\begin{remark}\label{rem:Teveniffvodd}
   For any vertex $v$ in a tournament $T$ we have that $$\charg(v)= \ind(v)-\out(v)=\ind(v)+\out(v)-2\out(v)=|V|-(2\out(v)+1).$$
In particular,  the number of vertices in a tournament is even if and only if some (equivalently, every) vertex has odd charge. 
\end{remark}

Our main focus is to construct and classify regular tournaments.

\begin{definition}
A tournament is called \textit{regular} if $\ind(v) = \out(v)$ for every vertex $v$.
\end{definition} 
A minimal tournament on more than one vertex is given by a directed triangle. In fact, we will often use these to balance out vertices with unassigned edges between them. A regular tournament can be defined in many analogous ways. In particular, note that the following are equivalent for any tournament $T=(V,E)$ on $n$ vertices.
\begin{enumerate}[label=(\roman*)]
    \item $T$ is regular.
    \item For any vertices $u,v\in V$, $\out(u)=\out(v)$.
    \item For any vertices $u,v\in V$, $\ind(u)=\ind(v)$.
    \item For any vertex $v\in V$, $\ind(v) =\frac{1}{2}(n-1)$.
    \item For any vertex $v\in V$, $\out(v) =\frac{1}{2}(n-1)$.
    \item For any vertex $v\in V$, $\charg(v)=0$.
\end{enumerate}

In particular, as then $\ind(v)+\out(v)=2\ind(v)=|V|-1$, any regular tournament is of odd order.

\section{Cofinality of Regular Tournaments}
Our first goal is to prove that any tournament can be embedded into a regular tournament. To this end, we first show that we can enlarge regular tournaments while preserving regularity.  
\begin{lemma}
Any regular tournament of order $n$ embeds in a regular tournament of order $n+2$.
\end{lemma}
\begin{proof} Let $T=(V,E)$ be a regular tournament of order $n$. We know that $n$ must be odd, say $n=2k+1$ for some integer $k$. Fix some vertex $v$ and partition the remaining vertices arbitrarily into two sets of equal cardinality, i.e. such that $V=X\cupdot Y \cupdot \{v\}$ where $|X|=k=|Y|$. We now add two new vertices $a$ and $b$ to $T$ and then complete $T\cup \{a,b\}$ into a regular tournament in the following way: For any $x\in X$ and $y\in Y$, complete $\{a,x,b,y\}$ into a directed square by adding the arrows $a\to x, x\to b, b\to y$ and $y\to a$. Further, complete $\{a,b,v\}$ into a directed triangle by adding $a\to v$, $v\to b$, and $b\to a$.

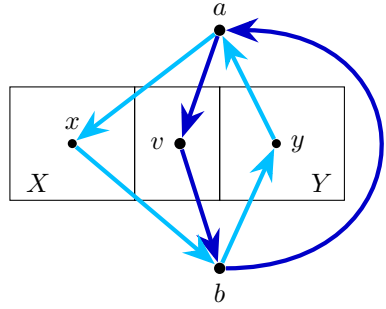
\begin{center}
 \begin{tikzpicture}[scale=0.75]
    \draw (0.3,0) rectangle (2.5,2);
    \node at (0.8, 0.3) {$X$};
    
    \draw (2.5,0) rectangle (4,2);
    \node[circle, fill, inner sep=1.5pt, label={left:$v$}] (v) at (3.3, 1) {};
    
    \draw (4,0) rectangle (6.2,2);
    \node at (5.8,0.3) {$Y$};

    \node[circle, fill, inner sep=1.5pt, label={above:$a$}] (a) at (4, 3) {};
    \node[circle, fill, inner sep=1.5pt, label={below:$b$}] (b) at (4, -1.2) {};
    
    \node[circle, fill, inner sep=1.2pt, label={above:$x$}] (x) at (1.4, 1) {};
    \node[circle, fill, inner sep=1.2pt, label={right:$y$}] (y) at (5, 1) {};
    \draw[arrowblue, -{Stealth[scale=1.2]}, ultra thick] (a) -- (v);
    \draw[arrowblue, -{Stealth[scale=1.2]}, ultra thick] (v) -- (b);
    
    \draw[arrowblue, -{Stealth[scale=1.2]}, ultra thick] 
        (b.east) to[out=0, in=0, looseness=2.2] (a.east);

    \draw[arrowcyan, -{Stealth[scale=1.2]}, ultra thick] (a) -- (x);
    \draw[arrowcyan, -{Stealth[scale=1.2]}, ultra thick] (x) -- (b);
    \draw[arrowcyan, -{Stealth[scale=1.2]}, ultra thick] (b) -- (y);
    \draw[arrowcyan, -{Stealth[scale=1.2]}, ultra thick] (y) -- (a);
\end{tikzpicture}
\end{center}





This defines a regular tournament on $V\cup \{a,b\}$ where every vertex has indegree and outdegree equal to $k+1$.
\end{proof}

\begin{corollary}\label{largeregulartournaments}
Let $T$ be a regular tournament of order $n$. For any odd integer $m\geq n$, there is a regular tournament of order $m$ that contains $T$ as a subtournament. 
\end{corollary}

\begin{theorem}
Every tournament is a subtournament of a regular tournament. 
\end{theorem}

\begin{proof}
We will prove the theorem by induction. The base case is clear. Fix some positive integer $n$ and assume that any tournament of order $n$ embeds into a regular tournament. Now, pick any tournament $T$ of order $n+1$ and fix some vertex $v$ of $T$. Let $m=\charg(v)$. Without loss of generality, assume $m\more 0$.

Let $T':=T\setminus \{v\}$ be the tournament obtained from $T$ by deleting the vertex $v$ and all arrows incident with it. By induction hypothesis and Corollary \ref{largeregulartournaments}, there exists a regular tournament $R$ of order at least $n+m$ containing $T'$ as a subtournament. Let $S:=R\setminus T'$ and notice that $S$ has at least $m$ vertices. 

Since $R$ is a regular tournament, we know it is of odd order, whence the following statements are equivalent: 
 \begin{itemize}[itemsep=2pt]
  \item [(1)] $S$ is of even order;
  \item [(2)] $T'$ is of odd order;
  \item [(3)] $T$ is of even order;
  \item [(4)] $m=\charg(v)$ is odd. 
 \end{itemize}

In particular, $\abs{S}-m$ is always odd.


 Fix a subset $B\subseteq S$ of $m$ many arbitrary vertices in $S$, say $B:=\{b_1, b_2,\ldots, b_m\}$, which we will use to balance the charge of $v$. By the above, the difference $S\setminus B$ is of odd size. Thus, after fixing an arbitrary vertex  $s_0\in S$, we can partition the vertices of $S\setminus (B\cup \{s_0\})$ into two sets $X$ and $Y$ of equal size, i.e.  
 $$S=B\cup X\cup Y\cup \{s_0\},$$ 

with $\abs{X}=\abs{Y}$. We thus have the two tournaments below, where the charge is denoted in green and $T'=\Ind(v)\cup \Out(v)$.

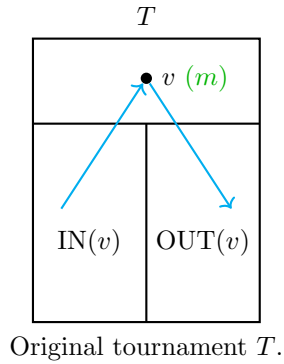
\begin{center}
\begin{minipage}[b]{0.4\textwidth}
\centering
\begin{tikzpicture}[scale = 0.75]
    \draw[thick] (0,0) -- (4,0) -- (4,5) -- (0,5) -- cycle;
    
    \draw[thick] (0,3.5) -- (4,3.5);
    
    \draw[thick] (2,0) -- (2,3.5);

    \node at (2, 5.4) { $T$};
    \node[circle, fill, inner sep=1.5pt, label={right:$v$ \textcolor{labelgreen}{$(m)$}}] (v) at (2, 4.3) {};
    
    \node at (1, 1.4) { $\Ind(v)$};
    \node at (3, 1.4) { $\Out(v)$};

    \draw[cyan, thick, ->] (0.5, 2) -- (v);
    \draw[cyan, thick, ->] (v) -- (3.5, 2);
\end{tikzpicture}

 Original tournament $T$.
\end{minipage}
    \hfill 
    \begin{minipage}[b]{0.55\textwidth}
        
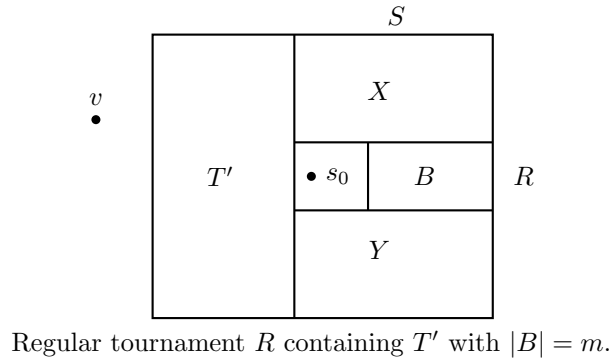
\centering
\begin{tikzpicture}[scale = 0.75]
    \def\rectH{5} 
    \def\rectW{6} 
    \def\midV{2.5}  
    \def\rowH{0.6} 

    \draw[thick] (0,0) rectangle (\rectW,\rectH);
    
    \draw[thick] (2.5,0) -- (2.5,\rectH);
    
    \draw[thick] (2.5, \midV - \rowH) -- (\rectW, \midV - \rowH); 
    \draw[thick] (2.5, \midV + \rowH) -- (\rectW, \midV + \rowH); 
    
    \draw[thick] (3.8, \midV - \rowH) -- (3.8, \midV + \rowH);

    \node at (1.2, \midV) {$T'$};
    \node[circle, fill, inner sep=1.2pt, label={above:$v$}] at (-1, 3.5) {};
    \node at (4, 4) {$X$};
    \node[circle, fill, inner sep=1.2pt, label={right:$s_0$}] at (2.8, \midV) {};
    \node at (4.8, \midV) {$B$};
    \node at (4, 1.2) {$Y$};
    
    \node[right] at (\rectW + 0.2, \midV) {$R$};
    \node[above] at (4.3, \rectH) {$S$};
\end{tikzpicture}

 Regular tournament $R$ containing $T'$ with $\abs{B}=m$.
    \end{minipage}
\end{center}


Our goal is, to amalgamate the two tournaments $T$ and $R$ over $T'$, and with the help of a new vertex $u$, complete this into a regular tournament. First note that in the free amalgam of $T$ and $R$ over $T'$, the charges are given as follows:
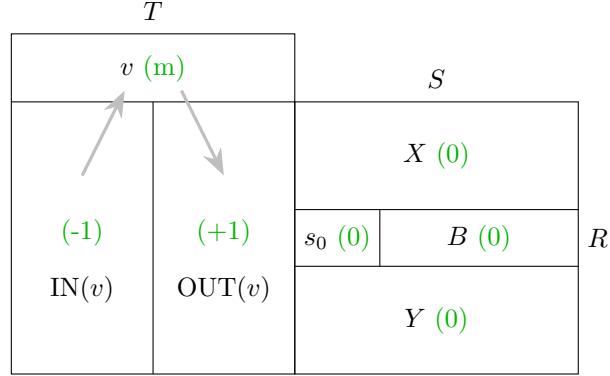
\begin{center}
    \begin{tikzpicture}[scale = 0.75]

    \draw (0,0) rectangle (5,6);
    \draw (0,4.8) -- (5,4.8); 
    \draw (2.5,0) -- (2.5,4.8); 

    \node at (2.5, 6.4) { $T$};
    \node at (2.5, 5.4) { $v$ \textcolor{labelgreen}{(m)}};
    
    \node at (1.25, 2.5) { \textcolor{labelgreen}{(-1)}};
    \node at (1.25, 1.5) { $\Ind(v)$};
    
    \node at (3.75, 2.5) { \textcolor{labelgreen}{(+1)}};
    \node at (3.75, 1.5) { $\Out(v)$};

    \draw[lightgray, -{Stealth[scale=1.2]}, very thick] (1.25, 3.5) -- (2.0, 5.0);
    \draw[lightgray, -{Stealth[scale=1.2]}, very thick] (3, 5.0) -- (3.75, 3.5);

    \draw (5,0) rectangle (10,4.8);
    
    \node at (7.5, 5.2) {$S$};
    \node[right] at (10, 2.4) {$R$};

    \draw (5, 1.9) -- (10, 1.9); 
    \draw (5, 2.9) -- (10, 2.9); 
    \draw (6.5, 1.9) -- (6.5, 2.9); 

    \node at (7.5, 3.85) { $X$ \textcolor{labelgreen}{(0)}};
    
    \node at (5.75, 2.4) {\normalsize $s_0$ \textcolor{labelgreen}{(0)}};
    \node at (8.25, 2.4) { $B$ \textcolor{labelgreen}{(0)}};
    
    \node at (7.5, 0.95) { $Y$ \textcolor{labelgreen}{(0)}};

\end{tikzpicture}
   
   Amalgam over $T'$, no edges between $v$ and $S$.
\end{center}

 
 First, we neutralize the charge of $v$ with help of our balancing vertices $B$. Recall that we assumed $\charg(v):=m$ in $T$ to be positive, whence we add an outgoing edge $v\to b$ for every $b\in B$. In the resulting graph, $v$ will have neutral charge then, while $\charg(b)=1$ for any $b\in B$. Furthermore, to produce a tournament, we need to introduce edges between $v$ and the remaining vertices in $X\cup Y\cup \{s_0\}$. 

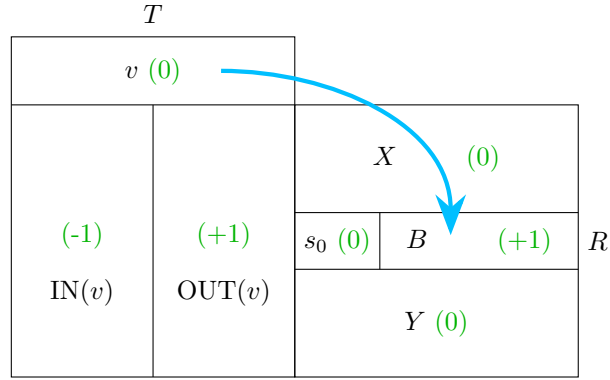
\begin{center}
    \begin{tikzpicture}[scale = 0.75]

    \draw (0,0) rectangle (5,6);
    \draw (0,4.8) -- (5,4.8); 
    \draw (2.5,0) -- (2.5,4.8); 

    \node at (2.5, 6.4) { $T$};
    \node (Vheader) at (2.5, 5.4) { $v$ \textcolor{labelgreen}{(0)}};
    
    \node at (1.25, 2.5) { \textcolor{labelgreen}{(-1)}};
    \node at (1.25, 1.5) { $\Ind(v)$};
    
    \node at (3.75, 2.5) { \textcolor{labelgreen}{(+1)}};
    \node at (3.75, 1.5) { $\Out(v)$};

    \draw (5,0) rectangle (10,4.8);

    \node[right] at (10, 2.4) { $R$};

    \draw (5, 1.9) -- (10, 1.9); 
    \draw (5, 2.9) -- (10, 2.9); 
    \draw (6.5, 1.9) -- (6.5, 2.9); 

    \node at (7.5, 3.85) { $X$ \hspace{0.8cm}\textcolor{labelgreen}{(0)}};
    \node at (5.75, 2.4) {\normalsize $s_0$ \textcolor{labelgreen}{(0)}};
    \node (Bbox) at (8.25, 2.4) { $B$\hspace{0.8cm} \textcolor{labelgreen}{(+1)}};
    \node at (7.5, 0.95) { $Y$ \textcolor{labelgreen}{(0)}};

    \draw[arrowcyan, -{Stealth[scale=1.5]}, ultra thick] 
        (3.7, 5.4) to[out=0, in=90] (7.75, 2.5);
\end{tikzpicture}

    Balancing $v$, no edges between $v$ and $S\setminus B$.
\end{center}


Next, we introduce one more vertex $u$ to the graph, which we will use to balance $T'$ and $B$ by introducing the opposite edges that $v$ had to them, i.e. 
 \[ \begin{cases} 
      u\to w & \text{ if } w\in \Ind_T(v), \\
      w\to u & \text{ if } w\in \Out_T(v)\cup B. 
   \end{cases}
\]

Note that then $\charg(u)=\abs{\Out_T(v)}+\abs{B}-\abs{\Ind_T(v)}=-\charg_T(v)+\abs{B}=-m+m=0$. Now every vertex is of neutral charge, while we still need to introduce arrows from $\{u,v\}$ to $X\cup Y\cup \{s_0\}$ and between $u$ and $v$ themselves.

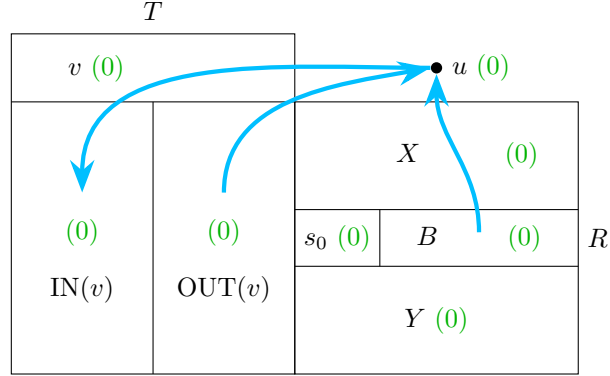
\begin{center}

    \begin{tikzpicture}[scale = 0.75]

    \draw (0,0) rectangle (5,6);
    \draw (0,4.8) -- (5,4.8); 
    \draw (2.5,0) -- (2.5,4.8); 

    \node at (2.5, 6.4) { $T$};
    \node (v_node) at (1.5, 5.4) { $v$ \textcolor{labelgreen}{(0)}};
    
    \node (inv) at (1.25, 2.5) { \textcolor{labelgreen}{(0)}};
    \node at (1.25, 1.5) { $\Ind(v)$};
    
    \node (outv) at (3.75, 2.5) { \textcolor{labelgreen}{(0)}};
    \node at (3.75, 1.5) { $\Out(v)$};

    \draw (5,0) rectangle (10,4.8);
    \node[right] at (10, 2.4) { $R$};

    \draw (5, 1.9) -- (10, 1.9); 
    \draw (5, 2.9) -- (10, 2.9); 
    \draw (6.5, 1.9) -- (6.5, 2.9); 

    \node at (7.5, 3.85) { \hspace{0.7cm} $X$ \hspace{1cm}\textcolor{labelgreen}{(0)}};
    \node at (5.75, 2.4) {\normalsize $s_0$ \textcolor{labelgreen}{(0)}};
    \node (B_label) at (8.25, 2.4) { $B$\hspace{0.8cm} \textcolor{labelgreen}{(0)}};
    \node at (7.5, 0.95) { $Y$ \textcolor{labelgreen}{(0)}};

    \node[circle, fill, inner sep=1.5pt, label={right: $u$ \textcolor{labelgreen}{(0)}}] (u_node) at (7.5, 5.4) {};

    
    \draw[arrowcyan, -{Stealth[scale=1.2]}, ultra thick] 
        (u_node.west) to[out=180, in=90] (1.25, 3.2);

    \draw[arrowcyan, -{Stealth[scale=1.2]}, ultra thick] 
        (3.75, 3.2) to[out=90, in=190] (u_node.west);

    \draw[arrowcyan, -{Stealth[scale=1.2]}, ultra thick] 
        (8.25, 2.5) to[out=90, in=270] (u_node.south);
\end{tikzpicture}
         
        $\charg(u)=-\charg(v)+\abs{B}=-m+m=0$\\
        no edges between $\{u,v\}$ and $S\setminus B$ and between $u$ and $v$.
\end{center}


 We complete the graph into a regular tournament, by completing $\{v,s_0,u\}$ into a directed triangle and for each $x\in X$ and $y\in Y$ we complete $\{v,y,u,x\}$ into a directed square by adding $v\to y, y\to u, u\to x$ and $x\to v$.



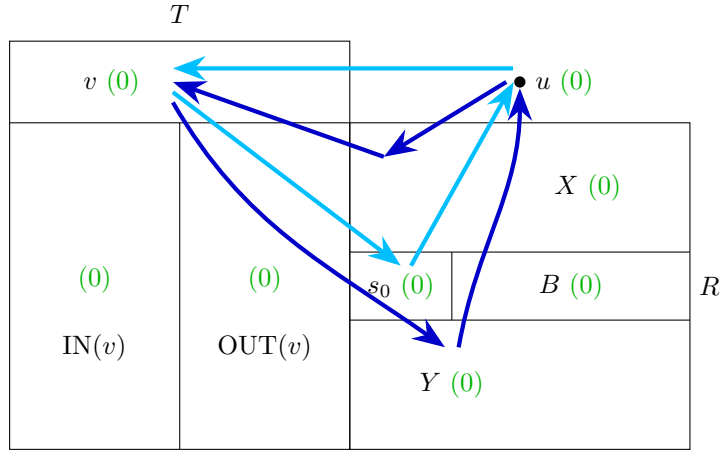
\begin{center}
\begin{tikzpicture}[scale = 0.9]
    \draw (0,0) rectangle (5,6);
    \draw (0,4.8) -- (5,4.8); 
    \draw (2.5,0) -- (2.5,4.8); 

    \node at (2.5, 6.4) { $T$};
    \node (v_node) at (1.5, 5.4) { $v$ \textcolor{labelgreen}{(0)}};
    
    \node (inv) at (1.25, 2.5) { \textcolor{labelgreen}{(0)}};
    \node at (1.25, 1.5) { $\Ind(v)$};
    
    \node (outv) at (3.75, 2.5) { \textcolor{labelgreen}{(0)}};
    \node at (3.75, 1.5) { $\Out(v)$};

    \draw (5,0) rectangle (10,4.8);
    \node[right] at (10, 2.4) { $R$};

    \draw (5, 1.9) -- (10, 1.9); 
    \draw (5, 2.9) -- (10, 2.9); 
    \draw (6.5, 1.9) -- (6.5, 2.9); 

    \node (X_label) at (8.5, 3.85) { $X$ \textcolor{labelgreen}{(0)}};
    \node (s0_label) at (5.75, 2.4) {\normalsize $s_0$ \textcolor{labelgreen}{(0)}};
    \node (B_label) at (8.25, 2.4) { $B$ \textcolor{labelgreen}{(0)}};
    \node (Y_label) at (6.5, 0.95) { $Y$ \textcolor{labelgreen}{(0)}};

    \node[circle, fill, inner sep=1.5pt, label={right: $u$ \textcolor{labelgreen}{(0)}}] (u_node) at (7.5, 5.4) {};

    \draw[arrowcyan, -{Stealth[scale=1.2]}, ultra thick] 
        (2.4, 5.25) to (5.75, 2.7);
    \draw[arrowcyan, -{Stealth[scale=1.2]}, ultra thick] 
        (5.9, 2.7) to (u_node.west);
    \draw[arrowcyan, -{Stealth[scale=1.2]}, ultra thick] 
        (7.4, 5.6) -- (2.4, 5.6);

    \draw[arrowblue, -{Stealth[scale=1.2]}, ultra thick] 
        (2.4, 5.1) to[out=-60, in=145] (6.4, 1.5);
    \draw[arrowblue, -{Stealth[scale=1.2]}, ultra thick] 
        (6.6, 1.5) to[out=80, in=-90] (u_node.south);
    \draw[arrowblue, -{Stealth[scale=1.2]}, ultra thick] 
        (7.3, 5.4) to (5.5, 4.3);
    \draw[arrowblue, -{Stealth[scale=1.2]}, ultra thick] 
        (5.5, 4.3) to (2.4, 5.4);
\end{tikzpicture}

Completing the tournament while keeping balance.
\end{center}

In the resulting tournament, now there is an arrow between any two vertices while the charge of any vertex is still neutral. Hence, we constructed a regular tournament containing $T$, as desired.
\end{proof}

\section{Distinguished Types of Regular Tournaments}

While the preceding theorem establishes that any tournament can be embedded into a regular one, the resulting structures are often purely existential and may lack internal symmetry. In the study of extremal problems and algebraic graph theory, one is frequently interested in tournaments that exhibit high degrees of structural regularity, such as rotational tournaments (or circulant tournaments). In such cases, the vertex set is identified with a cyclic group and the edge set is invariant under group rotation.

These structured tournaments are inherently regular, but they also satisfy deeper symmetries between the neighbourhoods of their vertices. Specifically, in many symmetric constructions, the in-neighborhood $\Ind(v)$ and out-neighborhood $\Out(v)$ are not merely equal in size, but are isomorphic or anti-isomorphic. To explore the limits of cofinality, we move beyond general regularity and ask whether every tournament can be found as a subtournament of these highly structured classes.

\begin{definition}
The \textit{inversion} of a tournament $T = (V, E)$ is the tournament $T^*$ obtained from $T$ by reversing all its arrows, that is, $T^*=(V,E^*)$ where 
\[ E^* = \{ (v,u) \mid (u,v) \in E \}. \]
\end{definition}

\begin{definition} 
Let $T=(V,E)$ be a regular tournament. We say $T$ is of \textbf{Type-I} if there exists a vertex $v \in V$ such that $\Ind(v)$ is isomorphic to $\Out(v)$. \\
On the other hand, we say $T$ is of \textbf{Type-II} if there exists a vertex $v \in V$ such that $\Ind(v)$ is isomorphic to the inversion of $\Out(v)$.
\end{definition}

Type-I tournaments represent a ``local symmetry" where the past and future of a vertex are structurally identical, while Type-II tournaments reflect a duality often found in Paley tournaments and quadratic residue constructions. We first show that Type-I regular tournaments are cofinal in the class of finite tournaments, demonstrating that the requirement of local isomorphism does not restrict the embedding power of the class.

We first show that Type-I regular tournaments are cofinal in the class of finite tournaments. 

\begin{lemma}
    Every tournament is a subtournament of a Type-I regular tournament. 
\end{lemma}
The idea here is very simple. Note that the term \textit{tournament} formalizes the idea of an actual sports tournament, where each player competes with every other one and edges indicate winning and loosing. In a regular tournament thus, each player wins as often as they loose. Given an arbitrary tournament outcome, if we could repeat the entire event, now ensuring each lost match is a win and vice versa, we would balance the score of every participant. Mathematically, this amounts to copying the vertex set and entering the described arrows. As now each player has to compete with their own copy as well, we also add a balancing vertex, which can be used to complete the two copies into a directed triangle. 

\begin{proof}
    Let $T=(V,E)$ be a tournament of order $n$ with $V=\{v_1,v_2,\ldots, v_n\}$. Let $S=(U,E')$ be a disjoint isomorphic copy of $T$ where $U=\{u_1,u_2,\ldots,u_n\}$. We construct a tournament $R$ of order $2n+1$ whose vertex set is $V\cup U \cup\{b\}$ where $b$ is a new vertex, and whose edge set is 
    $$E\cup E' \cup \{(u_j,v_i), (v_j,u_i)\mid (v_i,v_j)\in E\} $$
    $$\cup \{(v_i,w)\mid 1\leq i\leq n\}\cup\, \{(w,u_i)\mid 1\leq i\leq n\}\cup \{(u_i,v_i)\mid 1\leq i\leq n\} .$$

First note that indeed between any two vertices now there is a unique edge. This clearly holds if both vertices are in $V$, or both are in $U$ or one of them is $w$. For the leftover case, we consider $u_i\in U$ and $v_j\in V$ arbitrary. If $i=j$, then $(v_i,v_j)$ is not an edge in the original tournament, whence $u_i\to v_i$ is the unique edge between them. Otherwise, $i\neq j$, and either $(v_i,v_j)\in E$, whence in $R$ we have $v_j \to u_i$, or $(v_j,v_i)\in E$, whence in $R$ we add $u_i\to v_j$, i.e. if the $i$-th vertex "lost" in the original tournament, then now the $i$-th vertex "wins" and vice versa. This shows that there is exactly one directed edge between any two vertices in $R$. 

We now compute the indegrees in $R$. Clearly, $\ind(w)=n$ since $x\to w$ if and only if $x\in V$. Consider thus $1\leq i \leq n$ arbitrary and compute $\ind(v_i)$. .
\begin{eqnarray*}
\ind(v_i)&=&1+\ind_T(v_i)+|\{u_j\mid (u_j,v_i)\in E_R\}|\\
&=&1+\ind_T(v_i)+|\{v_j\mid (v_i,v_j)\in E\}|\\
&=&1+\ind_T(v_i)+\out_T(v_i)\\
&=&1+(n-1)=n.
\end{eqnarray*}
Analogously, we get $\ind(u_i)= n$, whence all vertices in $R$ have indegree equal to $n$, showing that $R$ is regular. Finally, observe that $\Ind(w)=V$ and $\Out(w)=U$, whence $\Out(w)\cong \Ind(w)$, since $U$ was chosen to be an isomorphic copy of $V$. Thus, $R$ is a regular tournament of Type-I containing $T$ as a subtournament.
\end{proof}

Next, we aim to show that Type-II regular tournaments are also cofinal. Towards this end, we will utilize the Gale-Ryser Theorem of matrices with integer coefficients.

\begin{definition}
Let $T=(V,E)$ be a tournament of order $n$ with $V=\{v_1,v_2,\ldots, v_n\}$. The \textit{adjacency matrix }of $T$ is the $n\times n$ matrix $A=[a_{ij}]$ defined as follows:
    \[
    a_{ij} = \begin{cases}
    1 & \text{if } v_i \to v_j,\\
    0 & \text{if } v_j \to v_i,\\
    0 & \text{if } i=j.
    \end{cases}
    \]
We call a square matrix a \textit{tournament matrix} if it is the adjacency matrix of some tournament.
\end{definition}

Clearly, $\out(v_i)$ is the sum of the $i$th row entries in $A$ and $\ind(v_j)$ is the sum of the $j$th column entries in $A$. Also observe that if $A$ is the adjacency matrix of a tournament $T$, then its transpose matrix, $A^\mathsf{T}$, is the adjacency matrix of the inversion tournament $T^*$. Moreover, an $n\times n$ matrix $A$ with entries from $\{0,1\}$ is a tournament matrix if and only if
    \[A + A^\mathsf{T} = J_n - I_n,\]
where, $I_n$ is the $n\times n$ identity matrix, and $J_n$ is the $n\times n$ matrix where every entry is $1$.

Suppose we are given a tournament $V$ of order $n$ with $V=\{v_1,v_2,\ldots, v_n\}$ and adjacency matrix $A$. We aim to construct a Type-II regular tournament $W$ of order $2n+1$ that contains $V$ as a subtournament. Let $U$ be a tournament with vertex set $U=\{u_1,u_2,\ldots,u_n\}$ such that $U\cong V^*$ and $V\cap U=\emptyset$. The adjacency matrix of $U$ is $A^\mathsf{T}$. 

The set the vertex set of $W$ to be $$W=\{v_1,v_2,\ldots, v_n,u_1,u_2,\ldots,u_n,w\}$$ where $w$ is a new vertex. Furthermore, we declare $w\to v_i$ and $u_i\to w$ for every $i\in\{1,2,\ldots, n\}$. Let $B$ be the adjacency matrix of $W$. The only missing data about $B$ is the entries between vertices of $V$ and vertices of $U$, that is, the submatrix $C$ below.

\[
B= \quad
\begin{NiceArray}{cccc|cccc|c}[first-row, first-col]
         & v_1 & v_2 & \cdots & v_n & u_1 & u_2 & \cdots & u_n & w \\
  v_1    & \Block{4-4}{\Huge A} &&&& \Block{4-4}{C} &&&& 0 \\ 
  v_2    & &&&&  &  &  &  & 0 \\
  \vdots & &&&&  &  &  &  & \vdots \\
  v_n    & &&&&  &  &  &  & 0 \\ \hline
  u_1    & \Block{4-4}{(J_n-C)^\mathsf{T}} &&&& \Block{4-4}{A^\mathsf{T}} &&&& 1 \\
  u_2    & &&&&  &  &  &  & 1 \\
  \vdots & &&&&  &  &  &  & \vdots  \\
  u_n    & &&&&  &  &  &  & 1 \\ \hline
  w      & 1 & 1 & \cdots & 1 & 0 & 0 & \cdots & 0 & 0 \\
  \CodeAfter
  \SubMatrix[{1-1}{9-9}][xshift=3pt]
\end{NiceArray}
\]

The problem now boils down to finding a submatrix $C$ that will turn $B$ into the adjacency matrix of a regular tournament. By design, such a regular tournament will be of Type-II since $\Ind(w)=U$ and $\Out(w)=V$, and $U\cong V^*$. The submatrix $C$ will only contain 0s and 1s, and we will set its mirror image about the main diagonal to be $(J_n-C)^\mathsf{T}$. This will ensure that $B$ is a tournament matrix. To ensure regularity, we need $\out(x)=n$ for all $x \in W$, i.e. the sum of all entries in any row of $B$ is $n$.

We now examine the entries of $C$. For $v_i \in V$ we want the sum of the entries in row $i$ of $B$ to equal $n$. This will happen if the sum of the $i$th row in $C$ is equal to $n-\out(v_i)$. For $u_j \in U$ we want the sum of the $(n+j)$th row of $B$ to equal $n$. This will happen if the sum of the $j$th row in $(J_n-C)^\mathsf{T}$ is equal to $n-1-\out(u_j)$, expanding this we get the following.

$$\sum_{i=1}^n \left((J_n-C)^\mathsf{T}\right)_{ji}=\sum_{i=1}^n (J_n-C)_{ij}=\sum_{i=1}^n 1-C_{ij}=n-\sum_{i=1}^n C_{ij}.$$

And so
\begin{equation*}
\renewcommand{\arraystretch}{1.5} 
\begin{array}{cccc} 
    & \sum_{i=1}^n \left((J_n-C)^\mathsf{T}\right)_{ji} & = & n-1-\out(u_j) \\  
    \text{iff} & n-\sum_{i=1}^n C_{ij} & = & n-1-\out(u_j) \\
    \text{iff} & \sum_{i=1}^n C_{ij} & = & 1+\out(u_j) \\
    \text{iff} & \sum_{i=1}^n C_{ij} & = & 1+\ind(v_j) \\
    \text{iff} & \sum_{i=1}^n C_{ij} & = & n-\out(v_j)
\end{array}
\end{equation*}

The last statement means that the sum of the entries of the $j$th column of $C$ is equal to $n-\out(v_j)$.
We thus want a matrix $C$ whose $i$th row sum is $n-\out(v_i)$ and whose $i$th column sum is also $n-\out(v_i)$. A characterization of such matrices is given by the Gale-Ryser Theorem.

For a given matrix $A=[a_{ij}]$, we define the \textit{row sum vector} to be the vector $(s_1,\dots, s_m)$, where $s_i:=\sum_{j=1}^na_{ij}$ and similarly for the \textit{column sum vector}.

\begin{theorem}[Gale-Ryser \cite{gale, ryser}]\label{gale-ryser}
Let $\lambda$ be a positive integer. Let $P = (p_1, p_2, \ldots , p_m)$ and $Q = (q_1, q_2 ,\ldots, q_n )$ be sequences of nonnegative integers with $q_1\geq q_2 \geq \cdots \geq q_n$ and where $p_1 + p_2 + \cdots + p_m = q_1 + q_2 + \cdots + q_n$.  There exists an $m$ by $n$ matrix $A = [a_{ij}]$ with integer entries $0\leq a_{ij} \le \lambda$ whose row sum vector is $P$ and column sum vector is $Q$ if and only if
   $$ \sum_{i=1}^k q_i \le \sum_{i=1}^{m} \min(p_i, \lambda k) \quad \text{for all } 1 \le k \le n.$$
\end{theorem}

Let $S=(s_1,\ldots,s_n)$ be the score sequence of a tournament $V$ on $n$ vertices. We define a new sequence $(r_i)$ given by $r_i=n-s_i$. In our context, we want to apply the Gale-Ryser Theorem with $\lambda=1$ and $P=Q=(r_1,r_2,\ldots,r_n).$ This sequence is clearly nonincreasing, as required in the Gale-Ryser Theorem. Hence a matrix $C$ whose row sum vector and column sum vector are equal to the above sequence $(r_i)$ exists if and only if

\begin{equation}\label{GaleRyserCriterion}\tag{$\dagger$}
\sum_{i=1}^k r_i \le \sum_{i=1}^{n} \min(r_i, k) \quad  \text{for all } 1 \le k \le n.
\end{equation}

Thus, we need to check for any tournament, that its sequence $(r_1,\ldots,r_n)$  satisfies Condition (\ref{GaleRyserCriterion}). Since $S$ is the score sequence of a tournament, it satisfies the Landau Condition (see \cite{landau}), that is, $\sum\limits_{i=1}^k s_i \ge \binom{k}{2}$ for every $1\leq k\leq n$. We now derive its equivalent condition on the $(r_i)$ sequence. For $1\leq k\leq n$, Landau implies that 
\begin{equation*}
    nk-\binom{k}{2} \ge nk-\sum_{i=1}^k s_i  = \sum_{i=1}^k (n-s_i) = \sum_{i=1}^k r_i. 
\end{equation*}
Therefore, the sequence $(r_i)$ satisfies
\begin{equation} \label{landau}\tag{$\star$}
    \sum_{i=1}^k r_i \le nk-\binom{k}{2}  \text{ for all } 1\leq k\leq n. 
\end{equation}

Consequently, for the sequence $(r_i)$ to satisfy Condition (\ref{GaleRyserCriterion}) it suffices to show that 
\begin{equation} \label{easier-condition}\tag{$\ddagger$}
    nk-\binom{k}{2}\leq \sum_{i=1}^n \min(r_i,k)  \text{ for all } 1\leq k\leq n. 
\end{equation}

To establish this inequality, we use the induction technique used in Griggs-Reid \cite{griggs}, which we will outline briefly below. But first we need to define an operation on score sequences.

\begin{definition}\label{defi:Griggs-ReidInduction}
Let $S:=(s_1,\ldots,s_n)$ be the score sequence of a nontransitive tournament on $n$ vertices. Let $t$ be the \textbf{smallest} index for which $s_t = s_{t+1}$ and define $m$ to be the number of occurrences of the integer $s_t$ in $S$. For $1 \le i \le n$, define a new sequence $S'=(s'_i)$ as follows:
\[
s'_i := \begin{cases}
    s_i-1 & \text{if } i=t; \\
    s_i+1 & \text{if } i=t+m-1; \\
    s_i & \text{otherwise.}\\
\end{cases}
\]

\end{definition}

We now describe the induction technique on score sequences, ref. \cite{griggs}. The base case of the induction is provided by the transitive tournament which has the score sequence $(0,1,\dots , n-1)$.  In the induction step, if $S$ is the score sequence of a nontransitive tournament $T$, we proceed by applying the operation above on $S$ which destroys blocks of identical entries, ensuring that eventually this operation leads to the transitive tournament. In \cite{griggs} it is shown that $S$ is a score sequence if and only if $S'$ is. This yields that a repeated reiteration of this process on any given score sequence must eventually lead to the transitive tournament, as the operation stops only when there are no repeated entries in the sequence and there is a unique score sequence with this property, namely, the one of the transitive tournament.   

Using the above, we now proceed with proving the main claim. We aim to show that a sequence satisfies the desired condition if and only if its primed sequence does. After first ensuring that the transitive tournament satisfies it, this yields that any score sequence does. By slight abuse of notation, when we say a score sequence $(s_i)$ satisfies Condition (\ref{easier-condition}), we mean its corresponding $(r_i)$ sequence satisfies Condition (\ref{easier-condition}).

\begin{theorem}
Let $(s_1,\ldots,s_n)$ be the score sequence of a tournament $V$ on $n$ vertices, and let $r_i=n-s_i$. Then the sequence $(r_i)$ satisfies Condition (\ref{easier-condition}).
\end{theorem}
\begin{proof} The base case of the Griggs-Reid induction consists of validating Condition (\ref{easier-condition}) for the transitive tournament of order $n$ with score sequence $(s_1, \ldots, s_n)=(0, 1, \ldots, n-1)$. Thus, for $r_i=n-s_i=n+1-i$, we obtain
$$(r_1, r_2, \cdots, r_n)=(n, n-1, \ldots, 3,2,1).$$
We proceed by induction on $k$. For the base case, we have 
$$\sum\limits_{i=1}^{n} \min(r_i, 1) = \sum\limits_{i=1}^{n} 1 = n \ge n\cdot 1-\binom{1}{2}=n.$$
Next, fix some $1 \leq k < n$ and  assume that
\begin{equation*}
    nk-\binom{k}{2} \le \sum_{i=1}^{n} \min(r_i, k). \tag{IH}
\end{equation*}
First, recall that $\binom{k+1}{2}=\binom{k}{2}+k$, whence $n(k+1)-\binom{k+1}{2} = (n-k)+nk-\binom{k}{2}$.  Next, observe that for all $r_i\leq k$ we have $\min(r_i,k)=\min(r_i,k+1)$, and for all $r_i> k$ we have $\min(r_i,k+1)=k+1=\min(r_i,k)+1$. Moreover, the second case occurs exactly in the first $n-k$ terms, namely, $r_1, r_2, \ldots, r_{n-k}$. Therefore, 
\begin{align*}
\sum_{i=1}^{n} \min(r_i, k+1)&= \sum_{i=1}^{n-k} \min(r_i, k+1) + \sum_{i=n-k+1}^{n} \min(r_i, k+1)  \\
&= \sum_{i=1}^{n-k} (\min(r_i, k)+1) + \sum_{i=n-k+1}^{n} \min(r_i, k)\\
&=(n-k)+\sum_{i=1}^{n} \min(r_i, k)\\
   &\overset{\text{(IH)}}{\geq} (n-k)+ nk-\binom{k}{2}\\
   &= n(k+1)-\binom{k+1}{2}.  
\end{align*}
This shows that transitive tournaments satisfy Condition (\ref{easier-condition}). 

Next, let $S=(s_1, s_2, \cdots, s_n)$ be the score sequence of a nontransitive tournament $T$, and as before, put $r_i=n-s_i$ and let $t$ be the \textbf{smallest} index for which $s_t = s_{t+1}$. Further, define $m$ to be the number of occurrences of the integer $s_t$ in $S$ and for $1 \le i \le n$. Following Griggs-Reid, we define $S'=(s'_i)$ via:
\[
s'_i := \begin{cases}
    s_i-1 & \text{if } i=t; \\
    s_i+1 & \text{if } i=t+m-1; \\
    s_i & \text{otherwise.}\\

\end{cases}
\]
Put $r'_i=n-s'_i$. 

We next prove an intermediary claim adding a stronger bound on Condition (\ref{easier-condition}) when $k=r_t$ (where $t$ is the smallest index for which $s_t = s_{t+1}$).

\textbf{Claim (A).} If $S$ satisfies Condition (\ref{easier-condition}), then
    $nr_t-\binom{r_t}{2} < \sum\limits_{i=1}^n \min(r_i,r_t).$
    
\textit{Proof of claim.}
    Since  $S$ satisfies Condition (\ref{easier-condition}) we get $\sum\limits_{i=1}^n \min(r_i,r_t) \geq nr_t-\binom{r_t}{2}$. Assume for the sake of contradiction, that 
    $\sum\limits_{i=1}^n \min(r_i,r_t) = nr_t-\binom{r_t}{2}.$ Note that $r_t\less n$, so $r_t+1\leq n$, whence we may use Condition (\ref{easier-condition}) on $r_t+1$ and get
    \[\sum_{i=1}^n \min(r_i,r_t+1) \ge n(r_t+1)-\binom{r_t+1}{2}.\]   
    Thus, $\sum\limits_{i=1}^n \min(r_i,r_t+1) - \sum\limits_{i=1}^n \min(r_i,r_t) \ge n-r_t,$ 
    and therefore,
    $|\{i \mid r_i > r_t\}|  \ge n-r_t.$  By choice of $t$, we have $r_1>r_2>\cdots>r_{t-1}>r_t$, and thus $|\{i \mid r_i > r_t\}| = t-1$. Therefore,  $r_t \ge n+1-t.$ From this and the fact $r_1>\cdots>r_{t-1}>r_t$, we get $r_i \ge n+1-i \text{ for } 1\leq i \le t.$
    
    Summing $(r_i)$ up to $t$, we get
    \[ \sum_{i=1}^t r_i \ge \sum_{i=1}^t (n+1-i) \ge nt - \binom{t}{2}.\]
    However, we know that $r_{t+1}=r_t$, and therefore,
    \[\sum_{i=1}^{t+1} r_i=\sum_{i=1}^{t} r_i +r_{t+1} \ge nt - \binom{t}{2}+n+1-t = n(t+1)-\binom{t+1}{2}+1>n(t+1)-\binom{t+1}{2}\]
    which contradicts Condition (\ref{landau}) for $k=t+1$. This is the end of the proof of Claim (A).   

We now prove the following.

\textbf{Claim (B).} The score sequence $S$ satisfies (\ref{easier-condition}) if and only if $S'$ satisfies (\ref{easier-condition}).

\textit{Proof of claim.} First, observe that by choice of $t$ we have the following.
 \[
\sum_{i=1}^n \min(r_i', k) = \begin{cases}
    \sum\limits_{i=1}^n \min(r_i, k) & \text{if } k\neq r_t; \\
    \sum\limits_{i=1}^n \min(r_i, k)-1 & \text{if } k=r_t.
\end{cases}
\]
For the reverse direction, assume $S'$ satisfies (\ref{easier-condition}) and note that in both cases above we have $ \sum\limits_{i=1}^n \min(r_i', k) \le  \sum\limits_{i=1}^n \min(r_i, k)$. We thus get that $S$ satisfies (\ref{easier-condition}). We now consider the forward direction. Assume $S$ satisfies (\ref{easier-condition}). When $k\neq r_t$, the sums remain the same and so (\ref{easier-condition}) is satisfied for $k\neq r_t$. When $k=r_t$, we have $\sum\limits_{i=1}^n \min(r_i', r_t)=\sum\limits_{i=1}^n \min(r_i, r_t)-1$, and by Claim (A), we get  $nr_t - \binom{r_t}{2} \le \sum\limits_{i=1}^n \min(r_i, r_t)-1$, and so $nr_t - \binom{r_t}{2} \le \sum\limits_{i=1}^n \min(r_i', r_t)$. Thus, $S'$ satisfies (\ref{easier-condition}). This is the end of the proof of Claim (B).   

We define a linear order $\prec$ on score sequences as follows: $(a_1, a_2, \cdots ,a_n) \prec  (b_1, b_2, \cdots ,b_n)$ iff $a_i < b_i$ where $i$ is the largest index with $a_i\neq b_i$. Note that, for any nontransitive score sequence $S$, we have $S \prec T_n$ where $T_n$ is the score sequence for the transitive $n$-tournament. Moreover, by \cite{griggs}, if $S$ is a score sequence, then $S'$ is also a score sequence. And by construction, $S \prec S'$.  So, by repeated application of this operation starting from the original score sequence $S$, we eventually reach $T_n$. And since $T_n$ satisfies (\ref{easier-condition}), we get $S$ also satisfies (\ref{easier-condition}) by Claim (B). Therefore, all score sequences satisfy (\ref{easier-condition}).
\end{proof}

\begin{corollary}
    Every tournament is a subtournament of a Type-II regular tournament. 
\end{corollary}
\begin{proof}
    Take any tournament $V=\{v_1,v_2,\ldots, v_n\}$ with score sequence $S$. Let $U$ be a tournament with vertex set $U=\{u_1,u_2,\ldots,u_n\}$ such that $U\cong V^*$ and $V\cap U=\emptyset$. We construct a regular tournament $W$ with vertex set $W=\{v_1,v_2,\ldots, v_n,u_1,u_2,\ldots,u_n,w\}$ where $w$ is a new vertex. We take all the directed edges of $V$ and $U$ and declare $w\to v_i$ and $u_i\to w$ for every $i\in\{1,2,\ldots, n\}$. By the above, $S$ satisfies Condition (\ref{easier-condition}), and so it satisfies Condition (\ref{GaleRyserCriterion}) and hence, by Gale-Ryser Theorem, there exists a matrix $C=[c_{ij}]$ whose $i$th row sum is $n-\out(v_i)$ and whose $i$th column sum is also $n-\out(v_i)$. We use $C$ to define the arrows between $V$ and $U$ as follows: 
    \[ \begin{cases} 
      v_i\to u_j & \text{ if } c_{ij}=1, \\
      u_j\to v_i & \text{ if } c_{ij}=0. 
   \end{cases}
\]
    By the properties of matrix $C$, and since $\Out(w)=V$ and $\Ind(w)=U\cong V^*$, it is clear that $W$ is regular tournament of Type-II containing $V$ as a subtournament.
\end{proof}

\section{A Constructive Algorithm for Type-II}

In this section, we present a greedy algorithm that, given an arbitrary tournament $V$, we conjecture produces a Type-II regular tournament in which $V$ embeds. The process begins with $V$, and extends it by adding a disjoint copy of its inversion tournament $U=V^*$ along with one additional new vertex $w$. Arrows are then added from $w$ to every vertex in $V$, and from each vertex in $U$ to $w$. Afterwards, the algorithm iterates over the vertices of $V$; where for each vertex in $V$, it defines its arrows to the vertices in $U$ in a greedy manner. Below, we give a detailed description of the algorithm.

\begin{tcolorbox}[colback=white, colframe=black!60, title=Greedy Algorithm for Constructing Regular Tournaments of Type-II]
\begin{itemize}[itemsep=4pt]
    \item Input: A tournament $V$ with $n$ vertices.
    \item Label the vertices of $V$ as $v_1, v_2, \ldots, v_n$, where $\out(v_i) \leq \out(v_{i+1})$.
    \item  Construct a tournament $U = \{u_1, u_2, \ldots, u_n\}$ with  $U\cap V=\emptyset$, and an edge set
    $\{(u_i,u_j)\mid v_j\to v_i\}.$ That is, $U$ is an inversion of $V$.
  \item Let $W=V\cup U \cup \{w\}$, where $w$ is a new vertex. 
  \item Define arrows between $w$ and the vertices in $V\cup U$ as follows.\\ For every $1\leq i\leq n$, define $u_i\to w$ and $w\to v_i$. 
\item Set $r_i=n-\out(v_i)$ for $1\leq i\leq n$.
\item Loop over the vertices of $V$ as described below.

\begin{enumerate}[topsep=5pt, itemsep=5pt]
    \item Set $i=1$. 
    \item WHILE $i\neq n+1$
    \begin{itemize}
    \item [(i)] For $1\leq j \leq r_i$, define $v_i \to u_j$ and $\ind(u_j)\texttt{++}$ . 
    \item [(ii)] For $r_i+1\leq j \leq n$, define $u_j \rightarrow v_i$.
    \item [(iii)] RELABEL the vertices of $U$ such that $\ind(u_j) \leq \ind(u_{j+1})$. 
    \item [(iv)] $i\texttt{++}$.
        \end{itemize}
\end{enumerate}
\end{itemize}
\end{tcolorbox}

We conjecture that this algorithm always produces a Type-II regular tournament containing $V$ as a subtournament. This has been verified computationally for all tournaments of order up to 17 which correspond to roughly $2.45\times 10^{26}$ many score sequences. 

\bibliographystyle{plain}
\bibliography{references}
\end{document}